\documentclass[11pt, amssymb,amsfonts]{amsart}
\usepackage{algorithm}
\usepackage{algpseudocode}

\usepackage{tikz}
\newcommand{\norm}[1]{ \left|  #1 \right| }
\newcommand{\Norm}[1]{ \left\|  #1 \right\| }

\def\P{{\hbox{\bf P}}}

 at 10 true pt

\def\rank{\hbox{\rm rank}}

\def\be#1{ \begin{equation}\label{#1} }

\def\bas{\begin{align*}}
\def\eas{\end{align*}}
\def\bi{\begin{itemize}}
\def\ei{\end{itemize}}

\def\emph#1{{\it #1}}
\def\textbf#1{{\bf #1}}

\theoremstyle{plain}
  \newtheorem{theorem}{Theorem}
  \numberwithin{theorem}{section}

  \newtheorem{lemma}{Lemma}
  \numberwithin{lemma}{subsection}
  \newtheorem{corollary}{Corollary}
   \numberwithin{corollary}{section}

\theoremstyle{remark}
  \newtheorem{remark}[subsubsection]{Remark}

\theoremstyle{definition}
  \newtheorem{definition}{Definition}
\numberwithin{definition}{section}
\include{psfig}

\title{Davis-Kahan Theorem under a moderate gap condition}
\author{Phuc Tran, Van Vu}

\thanks{Department of Mathematics, Yale University, 10 Hillhouse Ave, New Haven, Connecticut, USA; \\ \texttt{phuc.tran@yale.edu, van.vu@yale.edu} }

\date{}
\begin{document}
\maketitle

\begin{abstract} The classical Davis-Kahan theorem provides an efficient bound on the perturbation 
of eigenspaces of a matrix under a large (eigenvalue) gap condition. In this paper, we consider the case when the gap is moderate. Using a bootstrapping argument, we obtain a new bound which is efficient when the perturbation matrix is uncorrelated to the ground matrix. We believe that this bound is sharp up to a logarithmic term. \\


\end{abstract}

\section{The classical Davis-Kahan Theorem}

Consider a real, symmetric matrix $A$ of size $n$, with the spectral decomposition
 \begin{equation*}
    A  = \sum_{i = 1}^{n} \lambda_{i}u_{i}u_{i}^{\top},
\end{equation*}
where $\lambda_{1}  \ge  \lambda_{2}  \ge  \dots  \ge    \lambda_{n} $ are the eigenvalues of $u_i$ the corresponding eigenvectors. For a subset $S \subset \{1, \dots, n \}$, we denote
\begin{itemize}
    \item $\Pi_S:= \sum_{i \in S} u_i u_i^\top$ is the orthogonal projection onto the subspace spanned by the eigenvectors $u_i, i \in S$,
    \item $\Lambda_S:=\{ \lambda_i \,|\, i \in S\}$ and $\Lambda_{S^c}=\{ \lambda_j\,|\, j \in [n] \setminus S\}$.
   \end{itemize}
Let $E$ be a noise matrix and $\tilde A = A +E$. Define 
$\tilde \lambda_i $, $\tilde u_i, \tilde{\Lambda}_S, \tilde{\Lambda}_{S^c}$, and $\tilde{\Pi}_S$  respectively. In this paper, our goal is to estimate the difference
$$\| \tilde{\Pi}_S - \Pi_S \|,$$
where $\Norm{\cdot}$ denotes the spectral norm. We can extend our results to the asymmetric case via a simple symmetrization trick. 

 A classical result in numerical linear algebra, the Davis-Kahan theorem, is the standard tool to estimate the  perturbation of eigenspaces. This perturbation can also be defined as the sine of the angle between the two subspaces, and the Davis-Kahan theorem is often referred to as the Davis-Kahan sine theorem.

\begin{theorem}[Davis-Kahan \cite{DKoriginal, Book1}] \label{DKgeneral} Suppose that $\text{dist}\,(\Lambda_S, \tilde{\Lambda}_{S^c}) = \delta > 0$.  
Then, 
\begin{equation*}
    \| \tilde{\Pi}_S - \Pi_S \| \leq \frac{\pi \|E\|}{2 \delta}.
    \end{equation*}
Here $\text{dist}\,(\Lambda_S, \tilde{\Lambda}_{S^c}):= \displaystyle\min_{\lambda_i \in \Lambda_S, \tilde{\lambda}_j \in \tilde{\Lambda}_{S^c}}  |\lambda_i- \tilde{\lambda}_j| $.    
\end{theorem}
\begin{remark}
 It is well-known that the Davis-Kahan bound is sharp with the optimal constant factor $\pi/2$ (see \cite{Book1, Kato1, Mc1, SS1}).   
\end{remark}
 
In practice, we usually do not know the spectrum of the noisy matrix $\tilde A$.  It is desirable to use only the eigenvalues of $A$ in the estimate. A simple way is to use  Weyl's inequality, which gives 
\begin{equation*}
    \norm{\lambda_i - \tilde{\lambda}_{j}} \leq \norm{\lambda_i - \lambda_{j}} + \norm{\lambda_{j} -\tilde{\lambda}_{j}} \leq \norm{\lambda_i - \lambda_{j}} + \|E\|.
\end{equation*}
Therefore, replacing $\text{dist}\,(\Lambda_S, \tilde{\Lambda}_{S^c})$ by $\delta_S:= \min_{i \in S, j \notin S} |\lambda_i - \lambda_j|$ and paying a factor of $2$, one obtains the following corollary, which is more useful in applications. 
\begin{corollary} \label{DKS}  Let $S$ be a subset of $\{1,2, \cdots, n\},$ one has
\begin{equation}    \|\tilde{\Pi}_S - \Pi_S \| \leq \frac{ \pi \Norm{E}}{\delta_{S}}.
\end{equation}  
\end{corollary} 

Consider the important special case where $S=\{1,2, \cdots, p\}$, one obtains the following estimate for the perturbation of $\Pi_p$, the orthogonal projection onto the leading $p$-eigenspace of $A$.
\begin{corollary} \label{DKp} Let $\delta_p:= \lambda_p - \lambda_{p+1}$, one has 
\begin{equation}    \|\tilde{\Pi}_p - \Pi_p \| \leq \frac{ \pi \Norm{E}}{\delta_{p}}.
\end{equation}
    \end{corollary}

\begin{remark}\label{gaptonoiseRM}  (The large gap assumption) 
In applications of Theorem \ref{DKgeneral}, Corollary \ref{DKS}, or Corollary \ref{DKp}, 
we often need to bound the RHS by some small quantity $\epsilon$. 
This requires a {\it large eigenvalue gap} assumption 
\begin{equation} \label{gaptonoise} \delta_S \ge \delta \geq  \frac{\pi}{2\epsilon} \| E \|. \end{equation} 
For instance, if $\epsilon =.01$, 
the gap needs to be at least $50 \pi \|E \| \approx  157 \| E \|. $
In applications, when the size of the matrices tends to infinity, one often needs an asymptotic bound which tends to zero. In this case, 
the gap $\delta_S$ must be larger than $\| E \| $ by an order of magnitude. 
\end{remark}

The goal of this paper is to find effective bounds in the case where the gap is moderate, being only a small multiple of $\| E\|$ (say, $3 \| E\|$ or $4\| E\|$). 

The main idea of our analysis is to use a bootstrapping argument, combined with the moderate gap assumption,   to reduce the estimation of eigenspace perturbation to the estimation of a contour integral.  The evaluation of the integral is relatively simple and direct. This is different from the traditional contour integral approach; see the discussion at the beginning of Subsection \ref{subsec: contour argument}. 

Our new bounds will take into account the interaction between the noise matrix $E$ and the eigenvectors of $A$. This seems to be a very natural quantity to consider. Our bounds are effective when there is no strong correlation between $A$ and $E$. In particular, this is the case for most applications in data science, where noise is assumed to be random.   

The rest of the paper is organized as follows. In the next section, we present our new result. We will discuss the necessity of the terms in our new bounds and their sharpness, and make a comparison to recent developments in the field. In Section \ref{section: proof}, we describe the main idea which leads to a key lemma. In Section \ref{sec: lemma proof}, we prove our main theorem, using the key lemma and a series of technical 
lemmas. In Section \ref{section: applications}, we discuss an application in computer science, concerning the computation of the leading eigenvector of a large matrix. 
In Section \ref{section: prooflemma}, we provide the proofs of the technical lemmas used in Section \ref{sec: lemma proof}.

\section{New results}

 To ease the presentation, we first focus on the case of the eigenspace spanned by a few leading eigenvectors. 
 This is also one of the most important cases in applications. Another important case is when  
 \begin{equation} \label{S-sing} S=\{1,2, \dots, k, n- (p-k) +1, \dots, n \}, \end{equation} 
 for some properly chosen $k \le p$ such that the set $\{ \lambda_1, \dots, \lambda_k, -\lambda_{n-(p-k)+1}, \dots, -\lambda_n \}$ is the set of the largest $p$ singular values. In this case, we are talking about the projection on the subspace spanned by leading singular vectors.

 We fix a small index $p$ and consider the eigenspace spanned by the eigenvectors corresponding to 
the leading eigenvalues $\lambda_1, \dots, \lambda_p$. Notice that this space is well defined if $\lambda_p > \lambda_{p+1} $, even if there are indices  $i, j \le p$ where $\lambda_i = \lambda_j$. 
We denote by $\sigma_1$ the largest singular value of 
$A$,
$$\sigma_1 = \max_i |\lambda_i| = \|A\|. $$

\begin{theorem} \label{main1}  Let $r \ge p$ be the smallest integer such that $ \frac{\norm{\lambda_p}}{2} \leq \norm{\lambda_p - \lambda_{r+1}} $, and set  $x: =\max_{i,j \le r} | u_i^\top E u_j | $.
 Assume furthermore that  $$4\|E\| \leq \delta_p: = \lambda_p - \lambda_{p+1} \le \frac{|\lambda_p|} {4} . $$  Then 
\begin{equation}
\Norm{\tilde{\Pi}_p - \Pi_p} \leq 24 \left( \frac{\|E\|}{\norm{\lambda_p}} \log \left(\frac{6 \sigma_1}{\delta_p} \right) + \frac{r^2 x}{\delta_p} \right).
\end{equation}
\end{theorem}

\begin{remark}
    If we simply combine our contour bootstrapping argument in the next section with the trivial estimate of $F_1$ in \eqref{F1trivial}, then we get the Davis-Kahan bound (or at least up to some small constant). But now we proceed differently. It shows our new idea in a clearer way. 
\end{remark}

In the large gap case where $\delta_p  > \frac{| \lambda_p| }{4} $, 
Corollary 1.2 is better. Moreover, we need an upper bound on $\delta_p$ here to guarantee that the logarithmic term is positive. The constant 4 can be replaced by any constant larger than 2 (at the cost of increasing the constant 24 on the RHS). 

To compare this theorem to Corollary \ref{DKp}, let us ignore the logarithmic term and $r^2$ (which is typically small in applications). 
The RHS is thus simplified to 
$$ O \left( \frac{\|E\|}{\norm{\lambda_p}}  + \frac{ x}{\delta_p} \right). $$

In the first term on the RHS, we have $\| E \| $ in the numerator, but $\norm{\lambda_p} $ instead of $\delta _p =\lambda_p -\lambda_{p+1} $ in the denominator. Thus, we use the eigenvalue itself instead of the gap. 
The gap does not need to be large to make the bound effective as in the original Davis-Kahan theorem. This term will still be small if the leading eigenvalues themselves are large, but relatively close to each other.

In the second term, the denominator is $\delta_p$, as in the original Davis-Kahan theorem.   But here the numerator is $x$ rather than $\| E\|$. It is clear that $x \le \| E \|$. The key point, as we have already mentioned, is that if there is no strong correlation between $E$ and the leading eigenvectors of $A$, then $x$ can be much smaller than $\| E\|$. 

Let us illustrate this by considering the case where $E$ is random. This is a very common assumption for noise in real-life applications. Let $E$ be a Wigner matrix (a random symmetric matrix whose upper diagonal entries are i.i.d sub-Gaussian random variables with mean 0 and variance 1).

\begin{lemma}  Let $u, v$ be two fixed unit vectors and $E$ be a Wigner matrix.  With probability $1-o(1)$, $\| E \| = (2+o(1)) \sqrt n $ and $u^\top Ev = O(\log n)$. 
\end{lemma}

The first part of the lemma is a well-known fact in random matrix theory, while the second part is an easy consequence of the Chernoff bound (the proof is left as an exercise; see \cite{ H1, St1, SS1}). One can easily extend this lemma to other models of random matrices.

\begin{corollary}  \label{cor:lowrank} Assume that $A$ has rank $r$ and $E$ is Wigner. Assume  furthermore that $ \frac{|\lambda_p|}{4} \geq \delta_p \ge 8.01 \sqrt n$. Then with probability $1-o(1),$
$$ \Norm{\tilde{\Pi}_p - \Pi_p} =   O \left( \frac{\sqrt{n} }{\norm{\lambda_p}} + \frac{r^2 \log n }{\delta_p} \right). $$ 

\end{corollary}

This is superior to Theorem \ref{DKp}, if $| \lambda_p | \gg \delta_p $ and 
$\sqrt n \gg r^2$. In the case where $E$ is Gaussian, 
a similar result (under a weaker assumption) was obtained in \cite{OVK 13} by a different method. However, it seems highly non-trivial to extend the 
(fairly complicated) argument in \cite{OVK 22} to non-Gaussian matrices.

Finally, let us comment on the parameter $r$. If the matrix $A$ has low rank, then a convenient choice for $r$ is $r:= \rank (A)$ (as we did in Corollary \ref{cor:lowrank}). The low rank phenomenon has been observed and used very frequently in data science, to 
the degree that researchers have tried to find a theoretical explanation for why low rank matrices come up so often in practice \cite{UT1}. Among others, the low rank assumption is a key component in the solutions 
to the Netflix challenge problem, as well as the general matrix completion problem, one of the major topics in data science in recent years, has been based on this assumption; see \cite{DR1} for a survey.

In numerical applications, 
it has been observed that if the ground matrix $A$ does not have low rank, its spectrum often splits into two parts: one with few (say $r_0$) large eigenvalues, 
while the remaining ones are negligible \cite{KL1, RV1, RV2, UT1, Vbook, WW1}. In this case, the natural choice for $r$ is $r_0$.

{\bf \noindent Sharpness.}  We believe that under the given moderate gap assumption, our bound is sharp, up to a logarithmic factor. Consider Theorem \ref{main1} and assume that 
$r= O(1)$. In this case, we have 
$$ \| \tilde \Pi_p - \Pi_p \|  =  O \left( \frac {\| E \| }{|\lambda_p| }  + \frac{x} {\delta_p} \right).$$  
We believe that both terms on the RHS are necessary. First, the noise-to-signal 
term $\frac{\| E \| }{ |\lambda_p |} $ has to be present, since if the intensity of the noise $\| E \| $  is much larger than 
the signal $|\lambda_p|$, then one does not expect to keep the eigenvectors from the data matrix. One can see a concrete example in \cite{B-GN1}, which is known as the BBP threshold phenomenon in random matrix theory. In particular, it is shown, under some assumptions,  that if $E$ is a random Gaussian matrix, and $\| E \| \ge |\lambda_1 | $, then 
the first eigenvector $\tilde u_1$ of the perturbed matrix is completely random. This shows that one cannot expect any meaningful perturbation bound for $\|u_1 -\tilde u_1 \|$.

The second term $\frac{x}{ \delta_p }$ replaces the original bound $\frac{\| E \| }{\delta_p } $. The replacement of $\| E \|$  by $x$ is natural, as $x$ shows how $E$ interacts with the eigenvectors of $A$. We believe that this replacement is also optimal. 
In the random setting (Corollary \ref{main1}), $x$ is of order $ O( \log n)$ and it seems unlikely that one can replace this by $o(1)$. 

{\bf \noindent Related results.} 
There have been many extensions and improvements of Davis-Kahan theorem; see \cite{CCF1, EBW1, GTV1, HLMNV1, JW1, JW2, Kato1, KX1, KL1, Li1, OVK 13, OVK 22, SS1, Sun1, Vu1, Wu1, Z1}.

In most results we found in the literature,  both the assumption and the conclusion are different from ours, making a direct comparison impossible. The result that seems most relevant to ours is Theorem 1 of \cite{JW1}, a very recent paper, which applies to the case when $A$ is positive semi-definite. Compared to this result, both our assumption and bound have a simpler form. Furthermore, we do not require the (rather strong) restriction to positive semi-definiteness. Again, making a direct comparison regarding the strength of the bounds does not make too much sense here, unless one restricts to a very specific setting. 

Finally, let us mention that there are many works for the case when  $E$ is random \cite{EBW1, HLMNV1, KX1, OVK 13, OVK 22, Vu1, Wu1, Z1}. In this setting, the most relevant works are, perhaps, \cite{OVK 22} and \cite{KX1}. 
In \cite{OVK 22}, the authors considered $E$ to be a random Gaussian matrix (GOE), and basically proved the bound of Theorem \ref{main1}, even without the gap assumption. However, their argument uses special properties of Gaussian matrix and is restricted to this setting. 
In  \cite{KX1}, the authors also considered $E$ being GOE.  Among others, they showed (under some assumption) that with probability $1-o(1),$
$$\Norm{ \tilde{\Pi}_p - \Pi_p} = O \left[   \left(\frac{\|E\|}{\delta_{(p)}} \right)^2 + \frac{\sqrt{\log n}}{\delta_{(p)}} \right],$$
where $\delta_{(p)} = \min\{\delta_p, \delta_{p-1}\}.$

If we ignore the term $\frac{\sqrt{\log n}}{\delta_{(p)}}$ on the RHS, then this is a quadratic improvement over the original Davis-Kahan theorem 
(Theorem \ref{DKp}).  Our bound is sharper than this bound if  $\frac{\delta_p^2}{\lambda_p } \leq \|E\| $. In the moderate gap case (say $\delta_p= c \| E\|$ for some small constant $c$), this condition reduces to $c \delta_p \le \lambda_p$. 

The proof in \cite{KX1}  used strong properties of the Gaussian distribution and it seems hard to extend it to the other models of random matrices. On the other hand, 
our results apply to all models. 

{\bf \noindent Perturbation of the leading singular spaces.} Let $S$ be a set such that the eigenvalues with indices in $S$ have the largest absolute values (in other words, the set $\{| \lambda_i |, i \in S \} $ is the set of the $p$ largest singular values of $A$); see \eqref{S-sing}. Let $\sigma_1 \geq \cdots \geq \sigma_n$ be the singular values of $A$. We denote the projections onto the $p$-leading singular spaces of $A, \tilde{A}$ by $\Pi_{(p)},\, \tilde{\Pi}_{(p)}$ respectively.

We derive the ``halving distance"  $r$,  with respect to the two indices $(p,k)$,  as the smallest positive integer such that
$$ \frac{\lambda_k}{2} \leq \lambda_k - \lambda_{r+1} \,\, {\rm  and } \,\, \frac{ |\lambda_{n-(p-k)+1}|} {2} \leq \lambda_{n-r+1 } - \lambda_{n-(p-k)+1}.$$ 
Set 
$$\bar{x}:= \max_{ \substack{1 \leq i,j \leq r \\  n-r \le  i',j' \leq n }} \{  \norm{u_i^\top E u_j}, \norm{ u_{i'}^\top E u_{j'}} \}.$$
We obtain the following variant of Theorem \ref{main1} 

\begin{theorem}\label{main2.1} Assume that  $4\|E\| \leq \delta_k \leq \frac{\lambda_k}{4}$, $4\|E\| \leq \delta_{n-(p-k)} \leq \frac{\norm{\lambda_{n-(p-k)+1}}}{4}$, and $2\| E\| \leq \sigma_p -\sigma_{p+1}$. Then 
\begin{equation}
\Norm{\tilde{\Pi}_{(p)} -\Pi_{(p)}} \leq 48 \left( \frac{\|E\|}{\sigma_p} \log \left(\frac{6 \sigma_1}{\sqrt{\delta_k \delta_{n-(p-k)}}} \right) + \frac{r^2 x}{\min \{\delta_k, \delta_{n-(p-k)}\}} \right).
\end{equation} 

\end{theorem}  

The proof of this theorem only needs a minor modification from that of Theorem \ref{main1}, so we are going to leave it as an exercise. In a future paper, we will extend the method introduced here to treat the more general setting involving an arbitrary matrix functional (beyond eigenspace)  and a general set $S$. 

{\bf \noindent Rectangular matrices.}  We can extend Theorem \ref{main1} to non-symmetric, using a standard symmetrization trick.  
Consider a $m \times n$ matrix $A$ of rank $r_A \le \min \{m, n\}$. Consider the singular decomposition of $A$
  $$ A:= U \Sigma V^\top, $$ where $\Sigma = \text{diag}(\sigma_1,...,\sigma_{r_A})$ is a diagonal matrix containing the non-zero singular values $\sigma_1 \geq \sigma_2 \geq ... \geq \sigma_{r_A}$. The columns of the matrices $U=(u_1,..., u_{r_A})$ and $V=(v_1,...,v_{r_A})$ are the left and right singular vectors of $A$, respectively.  By definition 
$$U^\top U =V^\top V=I_{r_A}.$$
 For $1 \leq p < r_A$, we set  
\begin{equation}
\Pi_{p}^{left}=\sum_{1 \leq i \leq p} u_i u_i^\top \,\,\text{and}\,\, \Pi_{p}^{right} = \sum_{1 \leq i \leq p} v_i v_i^\top.
\end{equation}
Let $E$ be an $m \times n$ ``noise" matrix and set $\tilde{A}=A+E$. Define $\tilde{\Pi}_{p}^{left}, \tilde{\Pi}_{p}^{right}$ accordingly. We aim to bound 
$$\Norm{\tilde{\Pi}_{p}^{left} -\Pi_{p}^{left}}, \Norm{\tilde{\Pi}_{p}^{right}-\Pi_{p}^{right}}. $$
\noindent 
Set   $$\mathcal{A}:= \begin{pmatrix}
0 & A \\
A^\top & 0
\end{pmatrix}, \mathcal{E}:= \begin{pmatrix}
0 & E \\
E^\top & 0
\end{pmatrix}, \, \text{and hence} \,\, \tilde{\mathcal{A}}:= \begin{pmatrix}
0 & \tilde A \\
 \tilde A^\top & 0
\end{pmatrix}.$$ 

It is well known and easy to check that the eigenvalues of 
$\mathcal{A}$
are $\sigma_1 , \dots, \sigma_{r_A}, -\sigma_1 , \dots, -\sigma_{r_A}$ with corresponding eigenvectors 
$\frac{1}{\sqrt{2}}(u_1, v_1), \dots , \frac{1}{\sqrt{2}}(u_{r_A}, v_{r_A}), \frac{1}{\sqrt{2}}(u_1, -v_1), \dots,  \frac{1}{\sqrt{2}}(u_{r_A}, - v_{r_A} )$ (where $(u,v)$ is the concatenation of $u$ and $v$). Eigenvalues and eigenvectors of $\tilde{\mathcal{A}}$ behave similarly with respect to $\tilde{A}$. Thus, applying Theorem \ref{main1} on the pair $(\tilde{\mathcal{A}}, \mathcal{A})$ and the set of eigenvalues $\Lambda_{S'}:=\{\pm \sigma_1, \pm \sigma_2, \dots , \pm \sigma_p\}$, we obtain the following asymmetric version of Theorem \ref{main1}. 
\begin{theorem} \label{mainRec} Let $r \geq p$ be the smallest integer such that $\frac{\sigma_p}{2} \leq |\sigma_p - \sigma_{r+1}|$, and set $\bar{x}:= \max_{1 \leq i,j \leq r} \norm{u_i^\top E v_j}$. Assume furthermore that 
$$4\|E\| \leq \delta_p= \sigma_p - \sigma_{p+1} \leq \frac{\sigma_p}{4}.$$
Then, 
\begin{equation*}
    \max \left\lbrace \Norm{\tilde{\Pi}_{p}^{left} -\Pi_{p}^{left}}, \Norm{\tilde{\Pi}_{p}^{right}-\Pi_{p}^{right}} \right\rbrace \leq 24 \sqrt{2} \left( \frac{\|E\|}{\sigma_p} \log \left( \frac{6 \sigma_1}{\delta_p} \right) + \frac{r^2 x}{\delta_p}  \right).
\end{equation*}
\end{theorem}
The constant 24 gets replaced by $24 \sqrt 2$ due to the concatenation and usage of the Pythagorean theorem.

\section{The contour bootstrapping argument and a key lemma} \label{section: proof}
\subsection{The Contour Bootstrapping Argument } \label{subsec: contour argument}
 We start with the Cauchy theorem. Consider a contour $\Gamma$ and a point $a \notin \Gamma$, we have
 \begin{equation} \label{Cauchy0}  \frac{1} {2 \pi {\bf i }} \int_{\Gamma}  \frac{1 }{z-a} dz  = \begin{cases}
     & 1 \,\,\, \text{if  $a$ is inside $\Gamma$} \\
     & 0 \,\,\, \text{otherwise}
 \end{cases}. \end{equation} 
 Now, let  $\Gamma $ be a contour containing $\lambda_i, i \in \{1,2,...,p\}$ and assume that all $\lambda_j, j \notin \{1,2, \cdots, p\}$ are outside $\Gamma$. 
 We obtain the classical contour identity  
 \begin{equation} \label{contour-formula} 
  \frac{1} {2 \pi {\bf i }} \int_{\Gamma}  (zI-A)^{ -1} dz  = \sum_{ 1 \leq i \leq p} u_i u_i ^\top = \Pi_p; 
  \end{equation}  
 \noindent  see \cite{Kato1, SS1}.  
For a moment, assume that the eigenvalues $\tilde \lambda_i,  i\in \{1,2, \cdots, p\}$ are inside $\Gamma$, and all $\tilde \lambda_j, j \notin \{1,2, \cdots, p\}$ are outside. Then 
 \begin{equation} \label{contour-formula1} 
  \frac{1}{2 \pi {\bf i}}  \int_{\Gamma}  (zI-\tilde A)^{-1} dz  = \sum_{ 1 \leq i \leq p} \tilde u_i  \tilde u_i ^\top := \tilde{\Pi}_p .  \end{equation}  
\noindent This way, we derive a contour identity for the perturbation 
 \begin{equation} \label{f_Sperturbation formula}
\tilde\Pi_p- \Pi_p=  \frac{1} {2 \pi {\bf i} } \int_{\Gamma} [(zI-\tilde A)^{-1}- (zI- A)^{-1} ]  dz. 
  \end{equation}   
 \noindent Now we bound the perturbation  by the corresponding integral 
  \begin{equation} \label{Boostrapinequality1}
  \Norm{\tilde\Pi_p- \Pi_p} \leq \frac{1}{2 \pi} \int_{\Gamma} \Norm{ (zI-\tilde A)^{-1}- (zI- A)^{-1}} dz.  
  \end{equation}
  
 \noindent  The identity \eqref{f_Sperturbation formula} is well known, but our method to evaluate it is different from previous ones.  
 Traditionally, to control the RHS of \eqref{Boostrapinequality1},  researchers
 first use a series expansion for the RHS, and next control each term by analytical tools;  see for instance \cite[Part 2]{Kato1}. 
 
 We follow a different path, using a bootstrapping argument to reduce the estimation of the RHS of \eqref{Boostrapinequality1} to the estimation of a much simpler quantity, which can be computed directly. In what follows, we denote   $\frac{1}{2 \pi} \int_{\Gamma} \Norm{ [(zI-\tilde A)^{-1}- (zI- A)^{-1} ]} dz$ by $F$.
Using the resolvent formula
\begin{equation} M^{-1} - (M+N)^{-1}= (M+N)^{-1} N M^{-1}, \end{equation} and the fact that $\tilde A =A+E$, we obtain 
\begin{equation} (zI- A)^{-1} - (zI-\tilde A)^{-1} =  (zI-A)^{-1} E (zI-\tilde A)^{-1} .  \end{equation}

\noindent Therefore, we can rewrite $F$ as 
\begin{equation}
\begin{split}
F & = \frac{1}{2 \pi} \int_{\Gamma} \Norm{ (zI-A)^{-1} E (zI- \tilde{A})^{-1}} dz \\
&= \frac{1}{2 \pi} \int_{\Gamma} \Norm{ (zI-A)^{-1} E (zI-A)^{-1}  -   (zI-A)^{-1} E [(zI-A)^{-1} - (zI-\tilde{A})^{-1}]} dz.  
\end{split}
\end{equation}

\noindent Using triangle inequality, we obtain 
\begin{equation} \label{F(f,S)inequality1}
\begin{split}
F & \leq \frac{1}{2 \pi} \int_{\Gamma} \Norm{ (zI-A)^{-1} E (zI-A)^{-1}} dz + \frac{1}{2 \pi} \int_{\Gamma} \Norm{ (zI-A)^{-1} E [(zI-A)^{-1} - (zI-\tilde{A})^{-1}]}dz \\
&  \leq \frac{1}{2 \pi} \int_{\Gamma} \Norm{ (zI-A)^{-1} E (zI-A)^{-1}} dz  + \frac{1}{2 \pi} \int_{\Gamma} \Norm{ (zI-A)^{-1} E} \times \Norm{ [(zI-A)^{-1} - (zI-\tilde{A})^{-1}]}dz \\
& \leq \frac{1}{2 \pi} \int_{\Gamma} \Norm{ (zI-A)^{-1} E (zI-A)^{-1}} dz  +\frac{\max_{z \in \Gamma} \Norm{(zI-A)^{-1} E} }{2 \pi} \int_{\Gamma} \Norm{ [(zI-A)^{-1} - (zI-\tilde{A})^{-1}]}dz \\
& = \frac{1}{2 \pi} \int_{\Gamma} \Norm{ (zI-A)^{-1} E (zI-A)^{-1}} dz  +\max_{z \in \Gamma} \Norm{(zI-A)^{-1} E} \times F.
\end{split}
\end{equation}
Now we make an essential use of
our moderate gap assumption. We assume that the distance between any eigenvalue in $\{\lambda_1, \lambda_2, \cdots, \lambda_p\}$ and any eigenvalue in $\{\lambda_{p+1}, \lambda_{p+2},..., \lambda_n\}$ is at least $4  \| E \| $, namely
\begin{equation} \label{keyassumption} |\lambda_p - \lambda_{p+1}| \ge 4 \| E \| . \end{equation}  
Under this assumption, we can draw  $\Gamma$ such that the distance from any eigenvalue of $A$ to $\Gamma$ is at least $2 \| E\| $. 
(The simplest way to do so is to construct $\Gamma$ out of horizontal and vertical segments, where the vertical ones bisect the intervals connecting an eigenvalue in $S$ with its nearest neighbor (left or right) outside $S$.) With such  $\Gamma$, we have 

$$ \max_{z \in \Gamma} \Norm{(zI-A)^{-1} E} \leq  \frac{\|E\|}{2 \| E \| } = \frac{1}{2} \,(\text{by gap assumption}).$$

\noindent Together with \eqref{F(f,S)inequality1}, it follows that 
$$ F \leq F_1 + \frac{1}{2} F,$$
\noindent where  
$$F_1:= \frac{1}{2 \pi} \int_{\Gamma} \Norm{(zI-A)^{-1} E (zI-A)^{-1}} dz.$$


\noindent Therefore, 
\begin{equation} \label{Boostrapinequality2}
 \frac{1}{2} F \leq F_1.
\end{equation}
Notice that the gap assumption \eqref{keyassumption}  and Weyl's inequality ensure that $\tilde{\lambda}_i$ is inside the contour $\Gamma$ if and only if $i \in \{1,2, \cdots, p\}$. 
Combining \eqref{Boostrapinequality1} and \eqref{Boostrapinequality2}, we obtain our key inequality
\begin{lemma} \label{keylemma} Assume that $\delta_p \geq 4 \|E\|$, we have 
\begin{equation} \label{keybound}
\Norm{\tilde\Pi_p- \Pi_p} \leq 2 F_1.
\end{equation}

\end{lemma} 


\noindent This lemma is the heart of our  bootstrapping argument. The remaining task is to compute/estimate $F_1$.

\section {Estimating $F_1$ and the proof of Theorem \ref{main1} } \label{sec: lemma proof} 
Recall that $\sigma_1 = \max_{i \in [n]} |\lambda_i| = \|A\|$, and 
$$2 \pi F_1:= \int_{\Gamma} \Norm{(zI-A)^{-1} E (zI-A)^{-1}} dz.$$
\noindent Here we choose the contour $\Gamma$ to be a rectangle with vertices 
$$(x_0, T), (x_1, T), (x_1,-T), (x_0, -T),$$
\noindent where
 $$x_0 := \lambda_p - \delta_p/2, x_1:= 2 \sigma_1, T:= 2\sigma_1.$$

Now, we split $\Gamma$ into four segments: 
\begin{itemize}
\item $\Gamma_1:= \{ (x_0,t) | -T \leq t \leq T\}$. 
\item $\Gamma_2:= \{ (x, T) | x_0 \leq x \leq x_1\}$. 
\item $\Gamma_3:= \{ (x_1, t) | T \geq t \geq -T\}$. 
\item $\Gamma_4:= \{ (x, -T)| x_1 \geq x \geq x_0\}$. 
\end{itemize}
\usetikzlibrary{decorations.pathreplacing}
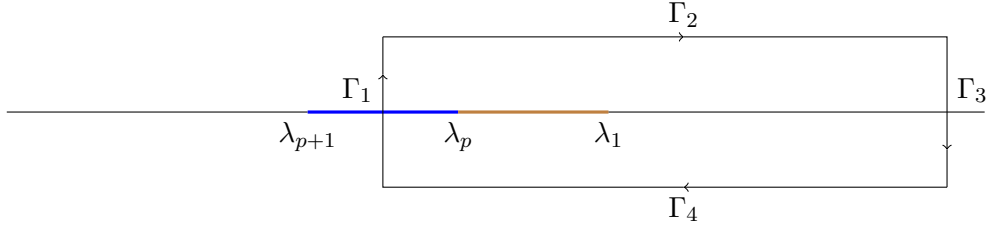
$$\begin{tikzpicture}

\coordinate (A) at (4,0);
\node[below] at (A){$\lambda_{p+1}$};
\coordinate (A') at (6, 0);
\node[below] at (A'){$\lambda_p$};
\coordinate (B) at (5,0);
\node[above left] at (B){$\Gamma_1$};
\coordinate (C) at (8,0);
\node[below] at (C){$\lambda_1$};
\coordinate (D) at (125mm,0);
\node[above right] at (D){$\Gamma_3$};
\coordinate (E) at (13,0);
\coordinate(B') at (5,5mm);
\coordinate (F) at (9,1);
\node[above] at (F){$\Gamma_2$};
\coordinate (G) at (9,-1);
\node[below] at (G){$\Gamma_4$};
\draw (0,0) -- (A);
\draw[very thick,blue] (A) -- (A');
\draw[very thick, brown] (A')  -- (C);
\draw (C) -- (D) -- (E);
\draw[->] (B) -- (B');
\draw (B') -- (5,1);

\draw [->] (5,1) -- (9,1);
\draw (9,1) -- (125mm,1) -- (D);
\draw [->] (D) -- (125mm, -5mm);
\draw (125mm,-5mm) -- (125mm,-1);
\draw [->] (125mm,-1) -- (9,-1); 
\draw (9,-1) -- (5,-1) -- (B);

\end{tikzpicture}.$$

Therefore, 
$$2\pi F_1 = \sum_{k=1}^4 M_k,\, M_k:= \int_{\Gamma_k} \Norm{(zI-A)^{-1} E (zI-A)^{-1} } dz.$$
We will use the following lemmas to bound $M_1, M_2, M_3, M_4$ from above. The proofs of these lemmas are delayed to Section \ref{section: prooflemma}.

\begin{lemma} \label{F1f1M1bound}
Under the assumption of Theorem \ref{main1}, 
\begin{equation}
M_1 \leq 70 \left( \frac{\|E\| }{\norm{\lambda_p}} \log \left(\frac{6 \sigma_1}{\delta_p} \right) + \frac{r^2 x}{\delta_p} \right) .
\end{equation}
\end{lemma}
\begin{lemma} \label{F1f1M2bound}
Under the assumption of Theorem \ref{main1}, 
\begin{equation}
M_2, M_4 \leq \frac{\|E\| \cdot \norm{x_1-x_0}}{T^2}. 
\end{equation}
\end{lemma}
\begin{lemma} \label{F1f1M3bound}
Under the assumption of Theorem \ref{main1}, 
\begin{equation}
M_3 \leq \frac{4\|E\|}{\norm{x_1-\lambda_1}}.
\end{equation}
\end{lemma}
By our setting of $x_1,x_0,T$,  we have 
\begin{equation}
    \begin{split}
 &   M_2,M_4 \leq \frac{\|E\| (|x_1|+|x_0|)}{4 \sigma_1^2} = \frac{\|E\|(2 \sigma_1 + |(\lambda_p + \lambda_{p+1})/2|)}{4 \sigma_1^2} \leq \|E\|\frac{ 3 \sigma_1}{4 \sigma_1^2} = \frac{3 \|E\|}{4 \sigma_1},\\
 & M_3 \leq \frac{4\|E\|}{|x_1| - \lambda_1} \leq \frac{4\|E\|}{\sigma_1}.
    \end{split}
\end{equation}

Therefore, 
\begin{equation} 
\begin{split}
F_1 & \leq \frac{1}{2\pi}\left( M_1+M_2+M_3+M_4 \right) \\
& \leq  \frac{70}{2\pi} \left[ \frac{\|E\|}{\norm{\lambda_p}} \log \left(\frac{6 \sigma_1}{\delta_p} \right) + \frac{r^2 x}{\delta_p} \right]  + \frac{3\|E\|}{ 4\pi \sigma_1} + \frac{4\|E\|}{2 \pi \sigma_1} \\
& = \frac{70}{2\pi} \left[ \frac{\|E\|}{\norm{\lambda_p}} \log \left(\frac{6 \sigma_1}{\delta_p} \right) + \frac{r^2 x}{\delta_p} \right]  + \frac{(11/2)\|E\|}{ 2\pi \sigma_1}  \\
& \leq \frac{(70+ 11/6)}{2\pi} \left[ \frac{\|E\|}{\norm{\lambda_p}} \log \left(\frac{6 \sigma_1}{\delta_p} \right) + \frac{r^2 x}{\delta_p} \right] \,\,\,(\text{since}\,\, \log \left(\frac{6 \sigma_1}{\delta_p} \right) > \log 24 > 3) \\
& \leq 12 \left[ \frac{\|E\|}{\norm{\lambda_p}} \log \left(\frac{6 \sigma_1}{\delta_p} \right) + \frac{r^2 x}{\delta_p} \right]. 
\end{split}
\end{equation}
Since $\Norm{\tilde{\Pi}_p -\Pi_p} \leq 2 F_1$, we finally obtain 
$$\Norm{\tilde{\Pi}_p -\Pi_p}  \leq 24 \left[ \frac{\|E\|}{\norm{\lambda_p}} \log \left(\frac{6 \sigma_1}{\delta_p} \right) + \frac{r^2 x}{\delta_p} \right],$$ completing the proof of Theorem \ref{main1}. 
\begin{remark} Following the proof of Lemma \ref{F1f1M2bound} in Section \ref{section: prooflemma}, one can obtain (somehow weaker) estimate on $M_1$:
$$M_1 \leq \frac{8\|E\|}{\delta_p}.$$
Hence, 
\begin{equation} \label{F1trivial}
  F_1 \leq \frac{1}{2 \pi} \left( \frac{8\|E\|}{\delta_p} + \frac{(11/2) \|E\|}{\sigma_1}\right) \leq \frac{8+11/4}{2 \pi} \times \frac{\|E\|}{\delta_p} < \frac{2\|E\|}{\delta_p}. 
\end{equation}

\end{remark}
\section{An application: Fast computation with random sparsification } \label{section: applications}
 
 %
 
 %

An important and well-known method in computer science to reduce 
the running time of important algorithms with input being a large matrix is to sparsify the input  (replacing a large part of the entries with zeroes). Intuitively, multiplying with zero simply costs no time, and multiplication is the key basic operation in most matrix algorithms. 
 The survey \cite{DriM1} discusses the (random) sparsification method in detail and contains a comprehensive list of references.

One way to sparsify the input matrix is to zero out each entry with probability $1-\rho$, independently, for some chosen small 
parameter $\rho$. Thus, from the input matrix $A$, we get a sparse matrix $A'$ with density $\rho$. 
The matrix $\tilde A =\frac{1}{\rho} A'$ is a sparse random matrix, whose expectation is exactly $A$. Now assume that we are interested in computing $f(A)$ for some function $f$. One then hopes that $f(\tilde A)$  would be a good approximation for $f(A)$. In summary, we gain on running time (thanks to the sparsification),  at the cost of 
the error term $\| f(\tilde A) -f(A) \| $.

Starting with the influential paper 
 \cite{AMc}, this procedure has been applied for a number of problems; 
 see, for example,  \cite{AFKMcS1, BDDMR1, BCN1, CLK1, CYZLK1, CmFmSq1, SKLZ1, SWZC1, SqCFT1, WMHZ1, XY1}.

The point here is that the new input $\tilde A$ can be written as $\tilde A =A+E$, where $E$ is a random matrix with zero mean.  This is exactly the favorite setting to apply our results. 
We are going to give an example of the problem of computing the leading eigenvector. Before stating our result, let us recall the definition of the \textit{stable rank} of matrix $A$ (see \cite[Section 7.6.1]{Vbook}). 
\begin{definition}[Stable Rank]
   Let $A$ be a matrix with singular values $\sigma_1 \ge \dots \ge \sigma_n$,  the stable  rank of $A$ is  
  $$r_{stable} : = \frac{\sum_{i=1}^n \sigma_ i } {\sigma_1 } . $$
\end{definition}
It is clear that if $r \ge h r_{stable} $,  for some $h \ge 1$, then $\sigma_r \le h^{-1/2} \sigma_1  $.   The assumption that a data matrix has a small stable rank is often used in recent applications; see \cite[Chapter 7]{Vbook}, \cite{KL1, RV1, RV2}.

\subsection{Computing the leading eigenvector} 
Let $A$ be a symmetric matrix with bounded entries $\| A \|_{\infty} \le K$. We want to estimate the first eigenvector of $A$, using the following classical algorithm \cite[Chapter 5]{NAbook}.

\begin{algorithm}[H]
\caption{Power Iteration}
\label{alg:thresholdandround}
\renewcommand{\thealgorithm}{}
\begin{algorithmic}[1]      
\State \text{Pick an initial unit vector } $\textbf{v}_0$.      
\State \text{For each natural number $k$, given unit vector $\textbf{v}_{k-1}$, compute} \, $\textbf{v}_{k}:= \frac{A \textbf{v}_{k-1}}{\Norm{A \textbf{v}_{k-1}}}.$ 
\State Return $\textbf{v}_N$ after $N$ iterations. \, 
\end{algorithmic}
\end{algorithm}

It is easy to prove that  $\textbf{v}_N$ is an estimate of $u_1$ with the error $O \left( \exp(- N \frac{\delta_1}{\lambda_2}) \right)$. 

Now we keep each pair (of symmetric) entries of $A$ (randomly and independently) with probability $\rho$, and zero out the rest. Let 
$A'$ be the resulting (symmetric)  matrix and
$\tilde A = \rho^{-1} A' $.

With the original (dense) input,  the 
running time for each iteration is $\Theta (n^2)$.
With the sparsified input, it reduces to $\Theta (n^2 \rho)$, 
which is significantly faster for small $\rho$, say, $\rho= n^{-1 + \epsilon } $. 
We now analyze the trade-off in accuracy, using our new result.

 Notice that each entry $a_{ij}$ has been replaced by a variable $\tilde a_{ij}$, which takes value $\rho^{-1} a_{ij}$ with probability $\rho$, and $0$ with probability $1- \rho$. Thus 
the expectation of $\tilde a_{ij}$ is exactly $a_{ij}$.  
 So we can write $\tilde A = A +E$, where $E$ is a random symmetric matrix with entries bounded by $K/\rho$ and zero mean.

We need the following estimates for random matrices: 
\begin{itemize}
\item  By \cite[Theorem 1.4 ]{Vu1}, if $\rho \gg \frac{\log^4 n}{n}$, then $\|E\| \leq 2K \sqrt{n/\rho}$ almost surely.
\item By \cite[Lemma 35]{OVK 13}, for any fixed unit vectors $u,v$, if $K/\rho > 1$, then 
$$\P \left(|u^\top E v| \geq  \frac{ 2\sqrt{2} K}{\sqrt{\rho}} t \right) \leq \exp \left( -t^2 \right).$$

\noindent This implies, by the union bound, that 
$$\P \left( x \geq  \frac{2 \sqrt{2} K}{\sqrt{\rho}} \log n \right) \leq r^2 n^{-2}, $$
where we define  $x: =\max_{i,j \le r} | u_i^\top E u_j | $.
\end{itemize}

\noindent Using these technical facts, it is easy to see that  Theorem \ref{main1} yields

\begin{corollary}  \label{sparse1} 
Let $r \ge 1$ be the smallest integer such that $ \frac{\norm{\lambda_1}}{2} \leq \norm{\lambda_1 - \lambda_{r+1}} $. Let $\rho$ be a positive number such that $1 \geq \rho \gg \frac{\log^4 n}{n}$. Assume furthermore that  $$8K \sqrt{n/\rho} \leq \delta_1 := \lambda_1 - \lambda_{2}  \leq \frac{\norm{\lambda_1}}{4}. $$  Then with high probability, 
\begin{equation} \label{sparse11} 
\Norm{\tilde{u}_1 - u_1} \leq \frac{72K}{\sqrt{\rho}} \left( \frac{ \sqrt{n}}{\norm{\lambda_1}} \log \left(\frac{6 \sigma_1}{\delta_1} \right) + \frac{ r^2 \log n }{ \delta_1} \right).
\end{equation}

\end{corollary}

It is clear that we can take $r= 4 r_{stable} (A) $. Corollary \ref{sparse1} is efficient in the case the stable rank is small. 
To make a comparison, let us notice that if we use Davis-Kahan bound (combined with the above estimate on $\| E \| $), we would obtain 
\begin{equation} \label{sparse12}  \Norm{\tilde{u}_1 - u_1} = O \left( \frac{\| E \| }{\delta_1} \right) = O \left( \frac{K \sqrt {n/ \rho} } {\delta_1 } \right).  \end{equation} 

To compare with Corollary \ref{sparse1}, notice that the first term on the RHS of \eqref{sparse11} is superior to the RHS of \eqref{sparse12} if $\lambda_1 \gg \delta_1$, which occurs if the two leading eigenvalues are relatively 
close to each other. The second term in the RHS of \eqref{sparse11} is also small compared to the RHS of \eqref{sparse12} if the stable rank is small, say,  $r_{stable} (A)= O(n^{1/4} )$. 


\section{Proofs of the lemmas} \label{section: prooflemma}
In this section, we prove the lemmas from Section 4. We first present a technical lemma, which will be used several times in the upcoming proofs.  
\begin{lemma} \label{ingegralcomputation1}
Let $a, T$ be positive numbers such that $a \leq T$. Then, 
\begin{equation}
\int_{-T}^{T} \frac{1}{t^2+a^2} dt \leq \frac{4}{a}.
\end{equation}
\end{lemma}
\begin{proof}[Proof of Lemma \ref{ingegralcomputation1}]
We have 
\begin{equation}
\begin{split}
\int_{-T}^{T} \frac{1}{t^2+a^2} dt & = \int_{-T}^{T} \frac{1}{a^2 u^2 +a^2} dt (\,t = au) \\
& = \frac{1}{a} \int_{-T/a}^{T/a} \frac{1}{u^2+1} du \\
& = \frac{2}{a} \int_{0}^{T/a} \frac{1}{u^2+1} du \\
& \leq \frac{2}{a} \left( \int_{0}^{1} \frac{1}{u^2+1} du + \int_{1}^{T/a} \frac{1}{u^2} du \right)  (\text{since}\,\,\frac{1}{1+u^2} \leq \frac{1}{u^2}) \\
& =  \frac{2}{a} \left( \int_{0}^{1} \frac{1}{u^2+1} du + 1 - \frac{a}{T} \right) \\
& \leq \frac{2}{a} \left( 1 + 1 - \frac{a}{T} \right) \leq \frac{4}{a}.
\end{split}
\end{equation}
\end{proof}
\subsection{Proofs of Lemma \ref{F1f1M2bound} and Lemma \ref{F1f1M3bound}} We denote $\{1,2,...,n\}$ by $[n]$. Notice that 
$$\Norm{(z-A)^{-1} E (z-A)^{-1}} \leq \frac{\|E\|}{\min_{i \in [n]}|z- \lambda_i|^2}.$$
Therefore, 
\begin{equation} \label{F_1f1inequality1}
M_2 \leq \int_{\Gamma_2}  \frac{1}{\min_{i \in [n]} |z-\lambda_i|^2 } \|E\| dz = \Norm{E} \int_{\Gamma_2}  \frac{1}{\min_{i \in [n]} |z-\lambda_i|^2 } dz.
\end{equation}
Moreover, since $\Gamma_2:= \{ z \,|\, z = x+ \textbf{i} T, x_0 \leq x \leq x_1 \},$
\begin{equation} \label{F_1f1inequality3}
\begin{split}
& \int_{\Gamma_2}   \frac{1}{\min_{i \in [n]} |z-\lambda_i|^2 } dz = \int_{x_0}^{x_1}  \frac{1}{ \min_{i \in [n]} ((x-\lambda_i)^2+T^2)} dx \leq \int_{x_0}^{x_1} \frac{1}{T^2} dx = \frac{|x_1 - x_0|}{T^2}.
\end{split}
\end{equation}
Therefore, $M_2 \leq \frac{\|E\| \cdot \norm{x_1-x_0}}{T^2}.$
Similarly, we also obtain that $M_4 \leq \frac{\|E\| \cdot \norm{x_1-x_0}}{T^2}.$\\
Next,
\begin{equation} \label{F_1f1inequality4}
\begin{split}
M_3 & \leq  \|E\| \int_{\Gamma_3}   \frac{1}{\min_{i \in [n]} |z-\lambda_i|^2 } dz \\
& = \|E\| \int_{-T}^\top  \frac{1}{ \min_{i \in [n]} ((x_1 - \lambda_i)^2 + t^2)} dt \,\,(\text{ since $\Gamma_3:=\{ z \,| \, z=x_1+ \textbf{i} t, -T \leq t \leq T$}) \\
& = \|E\| \int_{-T}^{T} \frac{1}{t^2 + (x_1-\lambda_1)^2} dt \\
& \leq \frac{ 4 \|E\|}{|x_1-\lambda_1|} \,(\text{by Lemma \ref{ingegralcomputation1}}).
\end{split}
\end{equation}

\subsection{Proof of Lemma \ref{F1f1M1bound}}
Using the spectral  decomposition  $(zI-A)^{-1} = \sum_{i=1}^n \frac{u_i u_i^\top}{(z- \lambda_i)}$, we can rewrite $M_1$ as 
\begin{equation}
M_1 = \int_{\Gamma_1} \Norm{ \sum_{n \geq i,j \geq 1} \frac{1}{(z-\lambda_i)(z-\lambda_j)} u_i u_i^\top E u_j u_j^\top} dz.
\end{equation}

\noindent Recall that $x:=\max_{1 \leq i,j \leq r} \norm{u_i^\top E u_j}$. Using the  triangle inequality, we have 
\begin{equation}
\begin{split}
M_1 & \leq \int_{\Gamma_1} \Norm{ \sum_{1 \leq i,j \leq r} \frac{1}{(z-\lambda_i)(z-\lambda_j)} u_i u_i^\top E u_j u_j^\top} dz + \int_{\Gamma_1} \Norm{ \sum_{r < i,j \leq n} \frac{1}{(z-\lambda_i)(z-\lambda_j)} u_i u_i^\top E u_j u_j^\top} dz \\
& + \int_{\Gamma_1} \Norm{ \sum_{\substack{i \leq r < j \\ i > r \geq j}} \frac{1}{(z-\lambda_i)(z-\lambda_j)} u_i u_i^\top E u_j u_j^\top} dz.
\end{split}
\end{equation}

\noindent Consider the first term, by triangle inequality, we have
\begin{equation}
\begin{split}
\int_{\Gamma_1} \Norm{ \sum_{1 \leq i,j \leq r} \frac{1}{(z-\lambda_i)(z-\lambda_j)} u_i u_i^\top E u_j u_j^\top} dz & \leq \sum_{1 \leq i,j \leq r} \int_{\Gamma_1} \Norm{ \frac{1}{(z-\lambda_i)(z-\lambda_j)} u_i u_i^\top E u_j u_j^\top}dz \\
& =  \sum_{1 \leq i,j \leq r} \int_{\Gamma_1} \frac{|u_i^\top E u_j| \cdot \|u_i u_j^\top\|}{\norm{(z-\lambda_i)(z-\lambda_j)}} dz  \\
& \leq \sum_{i,j \leq r} x \int_{-T}^\top \frac{1}{\sqrt{((x_0 -\lambda_i)^2+t^2)((x_0 -\lambda_j)^2+t^2)}} dt.
\end{split}
\end{equation}
The last inequality follows the facts that $x= \max_{1\leq i,j \leq r} |u_i^\top E u_j|, \|u_i u_j^\top\|=1$
 and $\Gamma_1:=\{ z\,|\,z= x_0 + \textbf{i} t, -T \leq t \leq T\}$. Moreover,  since 
\begin{equation} \label{x0lambdaiSmall}
    |x_0 - \lambda_i| \geq \delta_p/2 \,\,\, \text{for all}\,\, i \in [n],
\end{equation}
the RHS is at most
\begin{equation*}
    \begin{split}
    &  r^2 x \int_{-T}^{T} \frac{1}{t^2 +(\delta_p/2)^2} dt  \leq \frac{8 r^2 x}{\delta_p} \,\,\, (\text{by Lemma \ref{ingegralcomputation1}}).    
    \end{split}
\end{equation*}

Next, we apply the argument for bounding $M_2$  to estimate the second term:
\begin{equation}
\begin{split}
\int_{\Gamma_1} \Norm{ \sum_{n\geq i,j > r} \frac{1}{(z-\lambda_i)(z-\lambda_j)} u_i u_i^\top E u_j u_j^\top} dz & = \int_{\Gamma_1} \Norm{ \left(\sum_{n\geq i >r} \frac{u_iu_i^\top}{z- \lambda_i} \right) E \left( \sum_{n \geq i > r} \frac{u_i u_i^\top}{z -\lambda_i} \right)} dz \\
& \leq \int_{\Gamma_1} \Norm{\sum_{n \geq i >r} \frac{u_iu_i^\top}{z- \lambda_i} } \times \|E\| \times \Norm{\sum_{n \geq i >r} \frac{u_iu_i^\top}{z- \lambda_i} } dz\\
& \leq \int_{\Gamma_1} \frac{1}{\min_{n \geq i>r} |z- \lambda_i|} \times\|E\| \times \frac{1}{\min_{n \geq i>r} |z- \lambda_i|} dz \\
& = \|E\| \int_{\Gamma_1}  \frac{1}{ \min_{n \geq i>r} |z-\lambda_i|^2} dz \\
& \leq \|E\| \int_{-T}^{T} \frac{1}{\min_{n\geq i>r} ((x_0-\lambda_i)^2+t^2)} dt. 
\end{split}
\end{equation}
On the other hand, since $i > r$, 
\begin{equation} \label{x0lambdaiBig}
  \norm{x_0 -\lambda_i} = \norm{\lambda_p -\delta_p/2 - \lambda_{i}} \geq |\lambda_p - \lambda_{i}| - \frac{\delta_p}{2} \geq |\lambda_p - \lambda_{r+1}| - \frac{\delta_p}{2} \geq \frac{\norm{\lambda_p}}{2} -\frac{\delta_p}{2} \geq \frac{\norm{\lambda_p}}{4}.
\end{equation}
We further obtain
\begin{equation*}
    \begin{split}
 \int_{\Gamma_1} \Norm{ \sum_{n\geq i,j > r} \frac{1}{(z-\lambda_i)(z-\lambda_j)} u_i u_i^\top E u_j u_j^\top} dz & \leq  \|E\| \int_{-T}^{T} \frac{1}{t^2+ (\lambda_p/4)^2 } dt \\
 &  \leq \frac{16 \|E\|}{\norm{\lambda_p}} \,\,\,(\text{by Lemma \ref{ingegralcomputation1}}).
    \end{split}
\end{equation*}

\noindent Finally, we consider the last term:

\begin{equation} \label{lastterm1}
\begin{split}
 \int_{\Gamma_1} \Norm{ \sum_{\substack{i \leq r < j \\ i > r \geq j}} \frac{1}{(z-\lambda_i)(z-\lambda_j)} u_i u_i^\top E u_j u_j^\top} dz & \leq 2 \int_{\Gamma_1} \Norm{ \left(\sum_{1\leq i \leq r} \frac{u_iu_i^\top}{z- \lambda_i} \right) E \left( \sum_{n \geq j > r} \frac{u_j u_j^\top}{z -\lambda_j} \right)} dz \\
& \leq 2 \int_{\Gamma_1} \Norm{\sum_{1 \leq i \leq r} \frac{u_iu_i^\top}{z- \lambda_i} } \times \|E\| \times \Norm{\sum_{n \geq j >r} \frac{u_j u_j^\top}{z- \lambda_j} } dz\\
& \leq 2 \int_{\Gamma_1} \frac{1}{\min_{1 \leq i\leq r} |z- \lambda_i|} \times \|E\| \times \frac{1}{\min_{n \geq j>r} |z- \lambda_j|} dz \\ 
  & = 2 \|E\| \int_{\Gamma_1}  \frac{1}{\min_{i \leq r < j} \norm{(z-\lambda_i)(z-\lambda_j)}} dz.
   \end{split}
\end{equation}
By \eqref{x0lambdaiBig} and \eqref{x0lambdaiSmall}, the RHS is at most
\begin{equation*}
    \begin{split}
    & 2 \|E\| \int_{-T}^{T} \frac{1}{\sqrt{(t^2+(\delta_p/2)^2)(t^2+(\lambda_p/4)^2)}} dt = 4 \|E\| \int_{0}^{T}  \frac{1}{\sqrt{(t^2+(\delta_p/2)^2)(t^2+(\lambda_p/4)^2)}} dt.     
    \end{split}
\end{equation*}

\noindent Notice that  by Cauchy-Schwartz inequality, 

\begin{equation} \label{lasterm2}
\begin{split}
\int_{0}^{T}  \frac{dt}{\sqrt{(t^2+(\delta_p/2)^2)(t^2+(\lambda_p/4)^2)}}  & \leq \int_{0}^{T} \frac{2}{(t+\delta_p/2)(t+\norm{\lambda_p}/4)} dt \\
& =\frac{2}{\norm{\lambda_p}/4 -\delta_p/2} \int_{0}^{T} \left(\frac{1}{t+\delta_p/2} -\frac{1}{t+ \norm{\lambda_p}/4} \right)dt \\
& = \frac{2}{\norm{\lambda_p}/4 -\delta_p/2} \left[ \log \left( \frac{T+\delta_p/2}{\delta_p/2} \right) - \log \left( \frac{T+|\lambda_p/4|}{|\lambda_p/4|} \right)  \right]\\
& \leq \frac{16}{\norm{\lambda_p}} \times \log \left( \frac{2T+\delta_p}{\delta_p} \right) \,\,(\text{since}\,|\lambda_p|/4 -\delta_p/2 \geq |\lambda_p|/8).
\end{split}
\end{equation}

Together \eqref{lastterm1} and \eqref{lasterm2} imply that the last term is at most 
$$ \frac{64 \|E\|}{\norm{\lambda_p}} \times \log \left( \frac{2T+\delta_p}{\delta_p} \right).$$
These estimations imply that 
\begin{equation} 
M_1 \leq \frac{8r^2x}{\delta_p} + \frac{16 \|E\|}{\norm{\lambda_p}} + \frac{64\|E\|}{\norm{\lambda_p}} \times \log \left( \frac{2T+\delta_p}{\delta_p} \right).
\end{equation}
Since $T=2 \sigma_1 \geq 2 |\lambda_p| \geq 8 \delta_p$, thus $\frac{3T}{\delta_p} \geq 24$ and $\log  \left( \frac{3T}{\delta_p} \right) \geq \log 24 > 3$. We further obtain 
\begin{equation}\label{M_1F1bound}
M_1 \leq  \frac{8r^2x}{\delta_p} + \frac{6\|E\|}{\norm{\lambda_p}} \log \left( \frac{3T}{\delta_p} \right)+ \frac{64\|E\|}{\norm{\lambda_p}}  \log \left( \frac{2T+\delta_p}{\delta_p} \right) \leq 70 \left( \frac{\|E\|}{\norm{\lambda_p}} \log \left( \frac{3T}{\delta_p}\right) + \frac{r^2 x}{\delta_p} \right).
\end{equation}

\vskip3mm 
{\it Acknowledgment.} The research is partially supported by Simon Foundation award SFI-MPS-SFM-00006506 and NSF grant AWD 0010308.

\end{document}